\documentclass[12pt,reqno]{amsart}
\usepackage{geometry}
\usepackage[latin1]{inputenc}
\usepackage[italian,english]{babel}
\usepackage{amsmath, amsfonts, amsthm, mathrsfs}
\usepackage{paralist}
\usepackage{setspace}
\usepackage{graphicx}
\numberwithin{equation}{section}
\usepackage{fancyhdr,hyperref}

\newcommand{\R}{\mathbb{R}}

\DeclareMathOperator{\esssup}{ess \ sup}
\newtheorem{thm}{Theorem}[section]
\newtheorem{lem}[thm]{Lemma}

\newtheorem{prop}[thm]{Proposition}

\theoremstyle{definition}
\newtheorem{rem}[thm]{Remark}
\theoremstyle{remark}

\newenvironment{sistema}%
{\left\lbrace\begin{array}{@{}l@{}}}%
{\end{array}\right.}
\patchcmd{\abstract}{\scshape\abstractname}{\textbf{\abstractname}}{}{}
\makeatletter 
\def\@makefnmark{} 
\makeatother 

\title{\ \\ \vspace{0.4cm}
A nonlocal  anisotropic eigenvalue problem}

\begin{document}
\maketitle
\begin{center}
{\footnotesize GIANPAOLO PISCITELLI\\
Dipartimento di Matematica e Applicazioni \lq\lq R. Caccioppoli\rq\rq,\\
Universit\`a degli Studi di Napoli \lq\lq Federico II\rq\rq,\\
Complesso Universitario Montesantangelo, 
Via Cintia, 80126 Napoli, Italy}
\end{center}
\begin{abstract}
We determine the shape which minimizes, among domains with given measure, the first eigenvalue of the anisotropic laplacian perturbed by an integral of the unknown function. Using also some properties related to the associated \lq\lq twisted\rq\rq problem, we show that, this problem displays a \emph{saturation} phenomenon: the first eigenvalue increases with the weight up to a critical value and then remains constant.
\end{abstract}

\section{Introduction}
\thispagestyle{fancy}
Let $\Omega$ be a bounded open set in $\R^n$, with $n\geq 2$, and $H_0^1(\Omega)$ be the closure of the space $C^\infty_0 (\Omega)$ with respect to the norm $||u||:=(\int_\Omega|\nabla u|^2+ u^2)^{1/2}$. We consider the following minimization problem
\begin{equation}
\label{infquo}
\lambda (\alpha, \Omega )=\inf_{u \in H_0^1(\Omega)} \mathscr{Q}_\alpha (u,\Omega)
\end{equation}
with
\begin{equation}
\label{rayquo}
\mathscr{Q}_\alpha (u,\Omega)=\frac{\int_\Omega (H(\nabla u))^2\ \text{d}x+\alpha (\int_\Omega u\ \text{d}x)^2}{\int_\Omega u^2\ \text{d}x}
\end{equation}
where $\alpha$ is a real parameter and $H$ is a suitable homogeneous convex function. The minimization problem (\ref{infquo}) leads to the following eigenvalue problem
\begin{equation}
\label{eigpro}
\begin{sistema}
-\text{div} (H(\nabla u) \nabla H(\nabla u))+\alpha\int_\Omega u\ \text{d}x=\lambda u \quad \text{in} \ \Omega ,\\
u=0 \quad\text{on} \ \partial\Omega.
\end{sistema}
\end{equation}
In the euclidean case, when $H(\xi)=|\xi|$, problems like the above ones arise, for example, in the study of reaction-diffusion equations describing chemical processes (see \cite{S}). More examples can be found in \cite{BFNT}, \cite{CL}, \cite{D}, \cite{FI} and \cite{PI}. 

The extension to a general $H(\xi)$ is considered here as it has been made in other contexts to take into account a possible anisotropy of the problem. Typical examples are anisotropic elliptic equations (\cite{AFLT}, \cite{BFK}), anisotropic eigenvalue problems (\cite{DGI}, \cite{DGII}), anisotropic motion by mean curvature (\cite{BN}, \cite{BP}).

We also observe that, when $\alpha\rightarrow +\infty$, problem (\ref{infquo}) becomes a twisted problem in the form (see \cite{FH} for the euclidean case)
\begin{equation}
\lambda^T(\Omega)=\inf_{u\in H_0^1(\Omega)}\left\{ \frac{\int_\Omega (H(\nabla u))^2\ \text{d}x}{\int_\Omega u^2\ \text{d}x},\ \int_\Omega u\ \text{d}x=0 \right\}. 
\end{equation}
We denote by $\mathcal{W}$ the so-called Wulff set centered in the origin, that is the set \mbox{$\{x\in\R^n : H^o (x)<1 \}$}, where $H^o$ is polar to $H$. As in \cite{FH}, we prove the following isoperimetric inequality
\begin{equation}
\lambda^T(\Omega)\geq\lambda^T(\mathcal{W}_1 \cup \mathcal{W}_2),
\end{equation}
where $\mathcal{W}_1$ and $\mathcal{W}_2$ are two disjoint Wulff set, each one with measure $|\Omega|/2$.

The principal objective of this paper consists in finding an optimal domain $\Omega$ which minimizes $\lambda (\alpha, \cdot)$ among all bounded open sets with a given measure. If we denote with $\kappa_n$ (we refer to Section \ref{prelim} for details) the measure of $\mathcal{W}$, in the local case ($\alpha=0$) we have a Faber-Krahn type inequality
\begin{equation}
\lambda (0,\Omega)\geq \lambda (0,\Omega^\#)=\frac{\kappa_n^{2/n}j_{n/2-1,1}}{|\Omega|^{2/n}},
\end{equation}
where $j_{\nu,1}$ is the first positive zero of $J_\nu (z)$, the ordinary Bessel function of order $\nu$, and $\Omega^\#$  is the Wulff set centered at the origin with the same measure of $\Omega$. Hence, when $\alpha$ vanishes, 
the optimal domain is a Wulff set. We show that the non local term affects the minimizer of problem (\ref{infquo}) in the sense that, up to a critical value of $\alpha$, the minimizer is again a Wulff set, but, if $\alpha$ is big enough, the minimizer becomes the union of two disjoint Wulff sets of equal radii. This is a consequence of the fact that the problem (\ref{infquo}) have an unusual rescaling with respect to the domain. Indeed, we have
\begin{equation}
\lambda(\alpha, t \Omega)=\frac{1}{t^2}\lambda(t^{n+2}\alpha, \Omega),
\end{equation}
which, for $\alpha=0$, becomes
\begin{equation}
\lambda(0, t \Omega)=\frac{1}{t^2}\lambda(0, \Omega),
\end{equation}
that is the rescaling in the local case. Therefore we show that we have a Faber-Krahn-type inequality only up to a critical value. Above this, we show a saturation phenomenon (see \cite{FII} for another example), that is the estimate cannot be improved and the optimal value remains constant. More precisely, we prove the following
\begin{thm}
\label{thmsat}
For every $n\geq 2$, there exists a positive value
\begin{equation}
\alpha_c=\frac{2^{3/n}\kappa_n^{2/n}j^3_{n/2-1,1}J_{n/2-1,1}(2^{1/n}j_{n/2-1,1})}{2^{1/n}j_{n/2-1,1}J_{n/2-1}(2^{1/n}j_{n/2-1,1})-nJ_{n/2}(2^{1/n}j_{n/2-1,1})}
\end{equation}
such that, for every bounded, open set $\Omega$ in $\R^n$  and for every real number $\alpha$, it holds
\begin{equation}
\lambda (\alpha, \Omega)\geq
\begin{sistema}
\lambda(\alpha, \Omega^\#) \quad\text{if} \ \alpha |\Omega|^{1+2/n}\leq \alpha_c,\\
\frac{2^{2/n}\kappa_n^{2/n}j^2_{n/2-1,1}}{|\Omega|^{2/n}}\quad\text{if} \ \alpha |\Omega|^{1+2/n}\geq \alpha_c.
\end{sistema}
\end{equation}
If equality sign holds when $\alpha |\Omega|^{1+2/n}<\alpha_c$ then $\Omega$ is a Wulff set, while if inequality sign holds when $\alpha|\Omega|^{1+2/n}>\alpha_c$ then $\Omega$ is the union of two disjoint Wulff sets of equal measure.
\end{thm}
In Figure 1 we illustrate the transition between the two minimizers.
\begin{center}
\begin{picture}(300,200)
\put(52,173){\footnotesize $\lambda$}
\put(12,105){\footnotesize $2^{2/n} j^2_{\frac{n}{2}-1,1}$}
\put(30,77){\footnotesize $j^2_{\frac{n}{2}-1,1}$}
\put(50,35){\footnotesize O}
\put(120,35){\footnotesize $\alpha_c / \kappa_n^{1+2/n}$}
\put(240,35){\footnotesize $\alpha$}
\put(60,45){\includegraphics[width=0.5\textwidth]{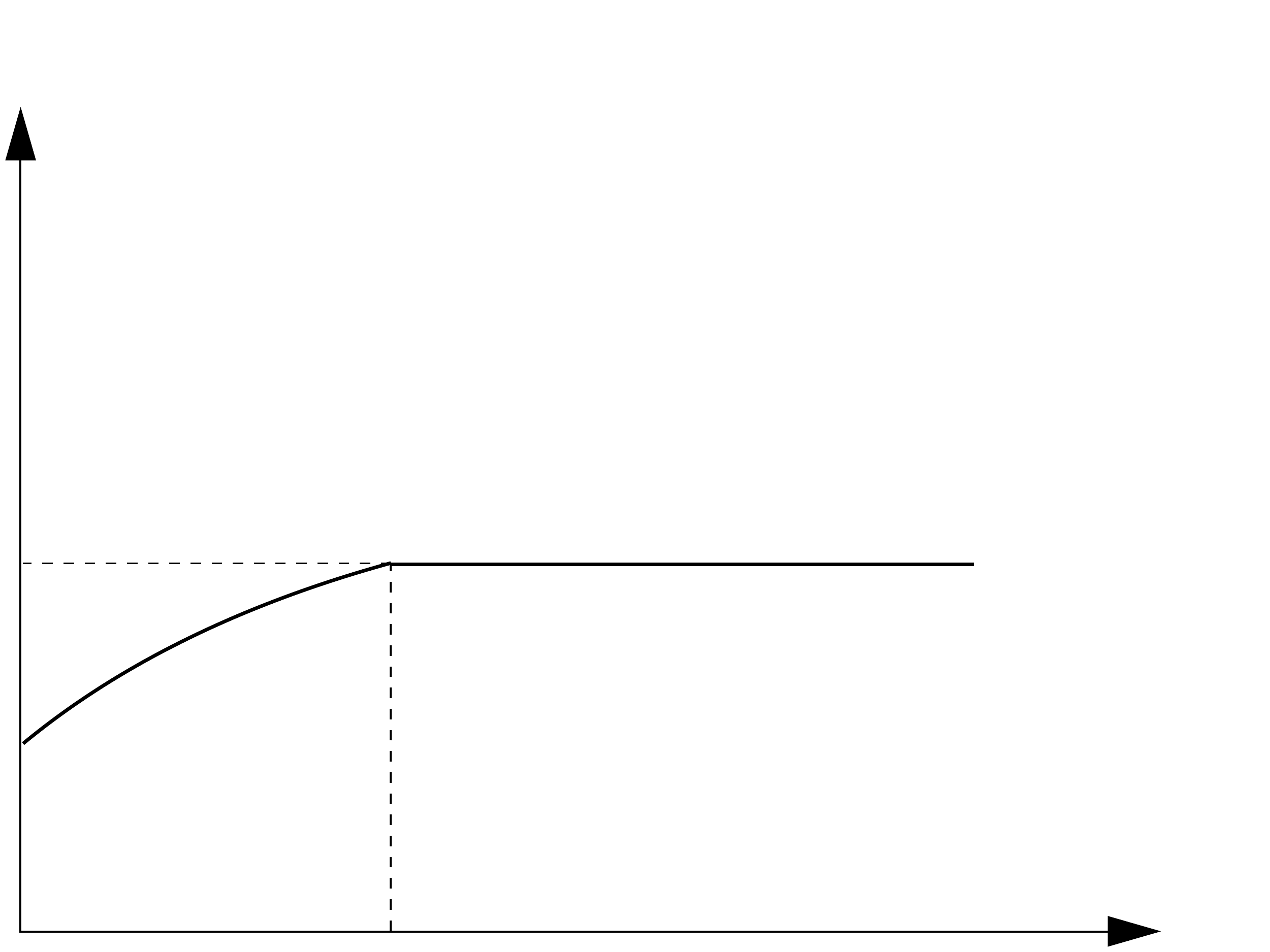}}
\put(-50,15){\footnotesize The continuous line represents the minimum of $\lambda (\alpha, \Omega)$, among the open bounded sets of }
\put(-50,5){\footnotesize measure $\kappa_n$, as a function of $\alpha$.}
\end{picture}
 \end{center}


The outline of the paper follows. In Section \ref{prelim} we recall some useful properties of gauge functions, rearrrangements and anisotropic Laplacian. In Section \ref{secpre} we show some properties of the first eigenvalue of (\ref{infquo}) and in Section \ref{sectwi} we investigate the first twisted Dirichlet eigenvalue. In Section \ref{secpro} we give the proof of the main theorem.

\section{Preliminaries}\label{prelim}

\subsection{Gauge functions}
\label{susega}
Let $H : \mathbb{R}^n \to [0,\infty[$ be a $C^1(\R^n\backslash\{ 0\})$ convex function satisfying the homogeneity property:
\begin{equation}
\label{hompro}
H(t\xi)=|t|H(\xi), \quad \forall \xi \in \mathbb{R}^n, \quad \forall t \in \mathbb{R},
\end{equation}
and such that any level set $\{ \xi\in\R^n : H(\xi)\leq t \}$ is strictly convex, for every $t>0$. Furthermore, assume that $H$ satisfies
\begin{equation}
\label{betabs}
\alpha | \xi |\leq H(\xi)\leq\beta|\xi|, \quad \forall \xi \in \mathbb{R}^n,
\end{equation}
for some positive constants $\alpha \leq \beta$. We also assume that
\begin{equation}
K=\{x \in \mathbb{R}^n : H(x)< 1\}
\end{equation}
has measure $|K|$ equal to the measure $\omega_n$ of the unit ball in $\mathbb{R}^n$. Because of ($\ref{hompro}$), this assumption is not restrictive and sometimes we will say that $H$ is the gauge of $K$.
We define the polar function $H^o: \R^n\rightarrow [0,+\infty[$ of $H$ as
\begin{equation}
H^o(x)= \sup_{\xi\in K}\frac{ \langle x, \xi \rangle}{H(\xi)}.
\end{equation}
It is easy to verify that also $H^o$ is a convex function which satisfies properties (\ref{hompro}) and (\ref{betabs}).
The set
\begin{equation}
\mathcal{W}= \{x\in \mathbb{R}^n: H^o(x)< 1\}
\end{equation}
is the so-called Wulff set centered at the origin and we indicate the Lebesgue measure of $\mathcal{W}$ by $\kappa_n$. More generally, we denote by $\mathcal{W}_r(x_0)$ the set $r\mathcal{W}+x_0$, that is the Wulff set centered in $x_0$ of radius $r$ and, throughout this paper, we put $\mathcal{W}_r:=\mathcal{W}_r(x_0)$ if no misunderstanding occurs. Let us observe that Lebesgue measure of $\mathcal{W}_r$ is $\kappa_nr^n$. Furthermore, we denote by $\Omega^\#$ the Wulff set centered at the origin such that $|\Omega^\#|=|\Omega|$.

The definitions of $H$ and $H^o$ give (see e.g. \cite{BP})
\begin{equation}
\label{anicom}
\begin{split}
H(\nabla H^o(x))&=1 \quad\text{and}\quad H^o(\nabla H(x))=1\quad\forall x \in\R^n;\\
H^o(x)\nabla H(\nabla H^o(x))&=x\quad\text{and}\quad H(x)\nabla H^o(\nabla H(x))= x \quad\forall x \in\R^n\backslash \{0\}.
\end{split}
\end{equation}
It is possible to give the following 
\lq\lq generalized\rq\rq definition of perimeter of a set $E$ with respect to $H$:
\begin{equation}
P_H(E; \Omega)=\int_\Omega |\nabla \chi_E|_H\ \text{d}x=\sup\left\{\int_\Omega  \text{div}   \varphi\ \text{d}x : \varphi\in C_0^1(\Omega; \mathbb{R}^n), H^o(\varphi)\leq 1  \right\}.
\end{equation}
This definition yields 
the \lq\lq generalized\rq\rq\ isoperimetric inequality (see \cite{AFLT}):
\begin{equation}
\label{isoine}
P_H(E; \mathbb{R}^n )\geq n \kappa_n^{1/n}|E|^{1-\frac{1}{n}}.
\end{equation}
We denote the generalized perimeter of $\mathcal{W}$ by $\gamma_n$ and hence we have that $\gamma_n=n\kappa_n$.

\subsection{Rearrangements}
Let $\Omega$ be a measurable and not negligible subset of $n$-dimensional euclidean space $\R^n$, let $u$ be a measurable map from $\Omega$ into $\R$.\\
We define the \emph{distribution function} of $u$ as the map $\mu$ from $[0,\infty[$ to $[0, \infty[$  such that
\begin{equation}
\mu (t) = |\{x \in \Omega : |u(x)|>t \}|
\end{equation}
and the \emph{decreasing rearrangement} of $u$, denoted by $u^*$, as the  map from $[0,\infty[$ to $[0, \infty[$  such that
\begin{equation}
u^*(s) :=\sup\{t>0 : \mu(t)>s \}.
\end{equation} 
For the properties of decreasing rearrangement we refer, for example, to \cite{K}. We define the (\emph{decreasing}) \emph{convex rearrangement} of $u$ (see \cite{AFLT}), denoted by $u^\#$, as the map such that
\begin{equation}
\label{conrea}
u^\#(x)=u^*(\kappa_n(H^o(x))^n).
\end{equation}
By definition it holds
\begin{equation}
||u||_{L^p(\Omega)}=||u^\#||_{L^p(\Omega^\#)},\quad \text{for} \ \ 1\leq p \leq +\infty.
\end{equation}
Furthermore, when $u$ coincides with its convex rearrangement, we have (see \cite{AFLT})
\begin{align}
\label{connau}
&\nabla u^\# (x)=u^{*'}(\kappa_n(H^o(x))^n)n\kappa_n (H^o(x))^{n-1}\nabla H^o(x);\\
\label{conhnu}
&H(\nabla u^\# (x))=-u^{*'}(\kappa_n(H^o(x))^n)n\kappa_n (H^o(x))^{n-1};\\
\label{connhn}
&\nabla H(\nabla u^\# (x))=\frac{x}{H^o(x)}.
\end{align}
Now, we recall here a result about a P\'olya-Szeg\"o principle related to $H$ (we refer to \cite{AFLT}, \cite{BZ}) in the equality case (see \cite{ET}, \cite{FV} for further details).
\begin{thm}
\label{psieth}
Let $u\in H_0^1(\Omega)$, then
\begin{equation}
\label{poszin}
\int_\Omega (H(\nabla u))^2\ \text{d}x\geq\int_{\Omega^\#} (H(\nabla u^\#))^2\ \text{d}s.
\end{equation}
Furthermore, if $u$ satisfies the equality in {\rm (\ref{poszin})}, then, for a.e. $t \in [0, \esssup u]$, the set $\{x \in\Omega : u(x)>t \}$ is equivalent to a Wulff set.
\end{thm}

\subsection{Properties of local problem}
Now we recall some known results about the anisotropic local ($\alpha = 0$) eigenvalue problem.
\begin{thm}
\label{fabthm}
Let $\Omega$  be an open bounded set, then
\begin{equation}
\label{fabkra}
\lambda (0,\Omega)\geq \lambda (0,\Omega^\#)=\frac{\kappa_n^{2/n}j_{n/2-1,1}}{|\Omega|^{2/n}}.
\end{equation}
\end{thm}
The details of the proof can be found in \cite[Th. 3.3]{BFK}. The computation of the first eigenvalue on $\Omega^\#$ comes from the fact that the first eigenfunction $u (x) = u^\# (x)$ in $\mathcal{W}_R$ satisfies (see also \cite{PII}):
\begin{equation}
\label{u*kpro}
\begin{sistema}
\frac{d^2}{d\rho^2} u^*(\kappa_n\rho^n)+ \frac{n-1}{\rho} \frac{d}{d\rho} u^*(\kappa_n\rho^n)+\lambda u^*(\kappa_n\rho^n)=0 \quad \text{in} \ \mathcal{W}_R \\
u^*(\kappa_n \rho^n)=0 \quad\text{on} \ \partial\mathcal{W}_R,
\end{sistema}
\end{equation}
where $\rho=H^o(x)$ and $R$ is the radius of the set $\Omega^\#$, which is a Wulff set.

As a consequence of these and other related Theorems, we have:
\begin{prop}
\label{tworem}
Let $\Omega$ be the union of two disjoint Wulff sets of radii $R_1, R_2\geq 0$.
\begin{itemize}
\item[(a)] If $R_1<R_2$, then  the first eigenvalue  $\lambda(0,\Omega)$ coincides with the first eigenvalue on the larger Wulff set. Hence, any associated eigenfunction is simple and identically zero on the smaller set and it does not change sign on the larger one.
\item[(b)] If $\Omega$ is the union of two disjoint Wulff sets of equal radii, then the first eigenvalue $\lambda (0,\Omega)$ is $\frac{2^{2/n}\kappa_n^{2/n}j^2_{n/2-1,1}}{|\Omega|^{2/n}}$. It is not simple and there exists an associated eigenfunction with zero average.
\end{itemize}
\end{prop}

\section{The first eigenvalue of the nonlocal problem}
\label{secpre}
In this Section we collect some properties of problem (\ref{infquo}), which will be fundamental in the proof of the main theorem.
\begin{prop}
Let $\Omega$ be an open bounded set, then the problem \text{\rm (\ref{infquo})} admits a solution $\forall\alpha\in\R$.
\end{prop}
\begin{proof}
The direct methods in the Calculus of Variation provide an existence proof for a minimizer of (\ref{infquo}). In a bounded domain $\Omega$, the existence of a first eigenfunction (and of the first eigenvalue) is established via a minimizing sequence $u_k$ for the Raylegh quotient. By homogeneity, it is possible the normalization and, using the Rellich-Kondrachov imbedding theorem \cite[Th. IX.16]{B}, we find a minimizer by the lower semicontinuity \cite[Th. 4.5]{G} of the functional.
\end{proof}

\begin{rem}
Let us note that if $u\in H_0^1(\Omega)$ is a minimizer of problem (\ref{infquo}), then it satisfies the associated Euler-Lagrange equation, that we can write as $L_\alpha u =\lambda u$, where
\begin{equation}
\label{operat}
L_\alpha u := -\text{\rm div}(H(\nabla u )\nabla H (\nabla u))+\alpha \int_\Omega u\ \text{d}x.
\end{equation}
\end{rem}

\begin{prop}
\label{anoalp}
Let $\Omega$ be a bounded open set which is union of two disjoint Wulff sets $\mathcal{W}_{R_1}(x_1)$ and $\mathcal{W}_{R_2}(x_2)$, with $R_1$, $R_2\geq 0$, and let $L_\alpha$ be the operator as in {\rm (\ref{operat})}. Then:
\begin{itemize}
\item[(a)] if a real number $\lambda$ is an eigenvalue of $L_\alpha u =\lambda u$ for some nonzero $\alpha$, either there exists no other real value of $\alpha$ for which $\lambda$ is an eigenvalue of $L_\alpha$ or $\lambda$ is an eigenvalue of the local problem \text{\rm ($\alpha=0$)}; in the last case $\lambda$ is an eigenvalue of $L_\alpha$ for all real $\alpha$. 
\item[(b)] $\lambda$ is an eigenvalue of $L_\alpha u =\lambda u$ for all $\alpha$ if and only if it is an eigenvalue of the local problem having an eigenfunction with zero average in $\Omega$.
\end{itemize}
\end{prop}
\begin{proof}
We set $\mathcal{W}_i$$=\mathcal{W}_{R_i}(x_i)$, $i=1,2$. We assume that $\lambda$ is an eigenvalue for two distinct parameters $\alpha_1$ and $\alpha_2$ and that $u$ and $v$ are the corresponding eigenfunctions. If we denote $u_i:=u|_{\mathcal{W}_i}$ and $v_i:=v|_{\mathcal{W}_i}$, $i=1,2$, then the functions $u_i$ satisfy
\begin{equation}
\label{a1uequ}
-\text{div} (H(\nabla u_i) \nabla H(\nabla u_i))+\alpha_1\left( \int_\Omega u\ \text{d}x  \right)=\lambda u_i \quad \text{on} \ \mathcal{W}_i, \ \text{for}\ i=1,2;
\end{equation}
and the functions $v_i$ satisfy
\begin{equation}
\label{a2vequ}
-\text{div} (H(\nabla v_i) \nabla H(\nabla v_i))+\alpha_2\left( \int_\Omega v\ \text{d}x  \right)=\lambda v_i \quad \text{on} \ \mathcal{W}_i, \ \text{for} \ i=1,2.
\end{equation}
We observe that $u_i(x)=u_i^*(\kappa_n (H^o(x-x_i))^{n})$ and $v_i(x)=v_i^*(\kappa_n (H^o(x-x_i))^n)$, $i=1,2$. This means that, by (\ref{connau}), (\ref{conhnu}) and (\ref{connhn}), we have  
\begin{equation}
\label{conequ}
\begin{split}
\int_{\mathcal{W}_i} H &(\nabla u_i) \nabla H(\nabla u_i)\nabla v_i \ \text{d}x \\
&= \int_{\mathcal{W}_i}-u_i^{*'}(\kappa_n(H^o(x-x_i))^n)  n\kappa_n (H^o(x-x_i))^{n-1}\frac{x-x_i}{H^o(x-x_i)}\cdot \\ 
&\qquad\qquad\qquad\qquad \cdot v_i^{*'}(\kappa_n(H^o(x-x_i))^n) n\kappa_n (H^o(x-x_i))^{n-1}\nabla H^o(x-x_i) \ \text{d}x\\ 
& =\int_{\mathcal{W}_i} H(\nabla v_i) \nabla H(\nabla v_i)\nabla u_i\ \text{d}x.
\end{split}
\end{equation}
for $i=1,2$. Now, we multiply the first equations of (\ref{a1uequ}) and (\ref{a2vequ}) respectively by $v_1$ and $u_1$, the second ones by $v_2$ and $u_2$ and then we integrate the first equations on $\mathcal{W}_1$ and the second ones on $\mathcal{W}_2$. By subtracting each one the equations integrated on $\mathcal{W}_1$ and using (\ref{conequ}), we get
\begin{equation}
\label{subfir}
\alpha_1\int_{\mathcal{W}_1} v_1\ \text{d}x \int_\Omega u \ \text{d}x  -\alpha_2  \int_{\mathcal{W}_1} u_1\ \text{d}x \int_\Omega v\ \text{d}x=0,
\end{equation}
in the same way we get also
\begin{equation}
\label{subsec}
\alpha_1 \int_{\mathcal{W}_2}v_2\ \text{d}x\int_\Omega u\ \text{d}x -\alpha_2  \int_{\mathcal{W}_2} u_2\ \text{d}x\int_\Omega v\ \text{d}x=0.
\end{equation}
Hence, the sum of (\ref{subfir}) and (\ref{subsec}) leads to
\begin{equation}
\label{2ei2al}
(\alpha_1-\alpha_2)\int_{\Omega} u\ \text{d}x\int_{\Omega} v\ \text{d}x=0.
\end{equation}
The result (a) follows because, if $\alpha_1$ and $\alpha_2$ are distinct, either $u_1$ or $u_2$ must have zero average, and hence satisfy the local equation. Finally, if (\ref{2ei2al}) is valid for all $\alpha_1$, $\alpha_2$, there is at least one eigenfunction with zero average and also (b) is proved.
\end{proof}

\begin{prop}
\label{sigthm}
Let $\Omega$ be a connected bounded open set and $\alpha \leq 0$. Then the first eigenvalue of \text{\rm (\ref{infquo})} is simple and the corresponding eigenfunction has constant sign in all $\Omega$.
\end{prop}
\begin{proof}
For any $u\in H_0^1(\Omega)$ we have $\mathscr{Q}_\alpha (u,\Omega) \geq \mathscr{Q}_\alpha (|u|,\Omega)$ with equality if and only if $u=|u|$ or $u= -|u|$. From now on, without loss of generality, we can assume that $u\geq 0$. Let us observe that if $u$ is a minimizer of (\ref{infquo}), then it satisfies (\ref{operat}) with $\alpha\int_\Omega u \leq 0$. Therefore 
$u$ is strictly positive in $\Omega$ by a weak Harnack inequality (see \cite[Th. 1.2]{T}). Now, we give a proof of simplicity following the arguments of  \cite{BFK} and \cite{BK}. Let $u$ and $v$ be two positive eigenfunctions, then we can find a real constant $c$ such that $u$ and $c v$ have the same integral:
\begin{equation}
\label{samint}
\int_\Omega u \ \text{d}x= \int_\Omega c \ v \ \text{d}x.
\end{equation}
We call $w$ the function $cv$, which is again an eigenfunction and we set
\begin{equation}
\label{semsom}
\varphi=\left(\frac{u^2+w^2}{2}\right)^{1/2}
\end{equation}
which is an admissible function. A short calculation yields
\begin{equation}
\nabla \varphi=\frac{\sqrt{2}}{2}\frac{u\nabla u+w\nabla w}{(u^2+w^2)^{1/2}}
\end{equation} 
and hence, by homogeneity, we have
\begin{equation}
\begin{split}
(H(\nabla \varphi))^2&=\frac{u^2+w^2}{2}\left(H\left(\frac{u\nabla u+w\nabla w}{u^2+w^2}\right)\right)^2\\
&=\frac{u^2+w^2}{2}\left(H\left(\frac{u^2\nabla\log u+w^2\nabla\log w}{u^2+w^2}\right)\right)^2.
\end{split}
\end{equation} 
Because of the convexity of $H(\xi)$ and the fact that $u^2/(u^2+w^2)$ and $w^2/(u^2+w^2)$ add up to $1$, we can use Jensen's inequality to obtain
\begin{equation}
\label{jenine}
\begin{split}
(H(\nabla\varphi))^2 & \leq\frac{u^2+w^2}{2}\left[\frac{u^2}{u^2+w^2}(H(\nabla\log u))^2+\frac{w^2}{u^2+w^2}(H(\nabla\log w))^2\right]\\
&=\frac{1}{2} (H(\nabla u))^2+\frac{1}{2} (H(\nabla w))^2.
\end{split}
\end{equation}
On the other hand, we have
\begin{equation}
\label{crlicp}
\left(\int_\Omega\varphi \ \text{d}x\right)^2\geq \left(\int_\Omega \left(\frac{u}{2}+\frac{w}{2} \right) \  \text{d}x\right)^2 =  \frac{1}{2}\left(\int_\Omega u \ \text{d}x\right)^2 +\frac{1}{2} \left(\int_\Omega w \ \text{d}x\right)^2
\end{equation}
Hence, definition (\ref{infquo}) and inequalities (\ref{jenine})-(\ref{crlicp}) yield the following inequality chain
\begin{equation}
\label{simine}
\begin{split}
\lambda(\alpha,\Omega)& \leq\frac{\int_\Omega (H(\nabla \varphi))^2\ \text{d}x+\alpha \left(\int_\Omega \varphi\ \text{d}x\right)^2}{\int_\Omega \varphi^2\ \text{d}x}\\
&\leq\frac{\frac{1}{2}\int_\Omega (H(\nabla u))^2\ \text{d}x+\frac{1}{2}\int_\Omega (H(\nabla w))^2\ \text{d}x+\frac{\alpha}{2} \left(\int_\Omega u\ \text{d}x\right)^2+\frac{\alpha}{2} \left(\int_\Omega w\ \text{d}x\right)^2}{\frac{1}{2}\int_\Omega u^2\ \text{d}x+\frac{1}{2}\int_\Omega w^2\ \text{d}x}\\
&=\lambda(\alpha,\Omega).
\end{split}
\end{equation}
Therefore, inequalities in (\ref{simine}) hold as equalities. This implies that $(H(\nabla \varphi))^2=\frac{1}{2} (H(\nabla u))^2+\frac{1}{2} (H(\nabla w))^2$ almost everywhere. By (\ref{jenine}), the strict convexity of the level sets of $H$ gives that $\nabla\log u =\nabla\log w$ a.e.. This proves that $u$ and $w$ are constant multiples of each other and, in view of (\ref{samint}), we have $u=w$. Therefore $u$ and $v$ are proportional.
\end{proof}


\begin{prop}
\label{eigalp}
Let $\Omega$ be a bounded open set, then:
\begin{itemize}
\item[(a)] the first eigenvalue of \text{\rm(\ref{infquo})} $\lambda(\alpha,\Omega)$ is Lipschitz continuous and non-decreasing with respect to $\alpha$ (increasing when the eigenfunction relative to $\lambda(\alpha,\Omega)$ has nonzero average);
\item[(b)] for nonnegative values of $\alpha$, the first eigenvalue of \text{\rm(\ref{infquo})} $\lambda(\alpha,\Omega)$ satisfies
\begin{equation}
\label{anfakr}
\lambda(\alpha,\Omega)\geq \frac{\kappa_n^{2/n}j^2_{n/2-1,1}}{|\Omega|^{2/n}};
\end{equation}
\item[(c)] for nonnegative values of $\alpha$, if $\Omega$ is the union of two disjoint Wulff sets of equal radii, the first eigenvalue of \text{\rm(\ref{infquo})} $\lambda(\alpha, \Omega)$ is equal to $\frac{2^{2/n}\kappa_n^{2/n}j^2_{n/2-1,1}}{|\Omega|^{2/n}}$.
\end{itemize}
\end{prop}
\begin{proof}
\noindent \begin{itemize}
\item[(a)] By simple computation we have the following inequalities
\begin{equation*}
\mathscr{Q}_\alpha (u,\Omega)\leq\mathscr{Q}_{\alpha+\varepsilon} (u,\Omega)\leq\mathscr{Q}_\alpha (u,\Omega)+|\Omega|\varepsilon\quad\forall \ \varepsilon>0.
\end{equation*}
Taking the minimum over all $u\in H_0^1(\Omega)$, we obtain
\begin{equation*}
\lambda (\alpha ,\Omega)\leq\lambda (\alpha+\varepsilon,\Omega)\leq\lambda (\alpha ,\Omega)+|\Omega|\varepsilon\quad\forall \ \varepsilon>0,
\end{equation*}
and, in view of Proposition \ref{anoalp}\emph{(a)-(b)}, the claim follows.
\item[(b)] By monotonicity of $\lambda(\alpha,\Omega)$ with respect to $\alpha$, we have that $\lambda(\alpha,\Omega)\geq\lambda (0, \Omega)$; then, by (\ref{fabkra}), we obtain the (\ref{anfakr}).
\item[(c)] By Proposition \ref{tworem}\emph{(b)}, if $\Omega$ is the union of two disjoint Wulff sets of equal radii, $\frac{2^{2/n}\kappa_n^{2/n}j^2_{n/2-1,1}}{|\Omega|^{2/n}}$ is the first eigenvalue of the local problem and it admits an eigenfunction with zero average. This implies that, by Proposition \ref{anoalp}\emph{(b)}, $\frac{2^{2/n}\kappa_n^{2/n}j^2_{n/2-1,1}}{|\Omega|^{2/n}}$ is an eigenvalue of $L_\alpha$ for all $\alpha$.
\end{itemize}
\end{proof}

\section{On the First Twisted Dirichlet Eigenvalue}
\label{sectwi}
In this Section we prove a Raylegh-Faber-Krahn type equation for the twisted eigenvalue problem
\begin{equation}
\label{inftwi}
\lambda^T(\Omega)=\inf_{\substack{u\in H_0^1(\Omega) \\ u \not\equiv 0}}\mathscr{Q}^T(u,\Omega),
\end{equation}
where
\begin{equation}
\mathscr{Q}^T(u,\Omega)=\left\{ \frac{\int_\Omega (H(\nabla u))^2\ \text{d}x}{\int_\Omega u^2\ \text{d}x},\ \int_\Omega u\ \text{d}x=0 \right\}. 
\end{equation}
Let us denote by
\begin{equation}
\label{deo+o-}
\Omega_+=\{x\in\Omega, u(x)>0\}\quad\text{and}\quad\Omega_-=\{x\in\Omega, u(x)<0\},
\end{equation}
and by $\mathcal{W}_+$ and $\mathcal{W}_-$ the Wulff sets such that $|\mathcal{W}_\pm|=|\Omega_\pm|$.

\begin{lem}
\label{lemudb}
\begin{equation*}
\lambda^T(\Omega)\geq\lambda^T(\mathcal{W}_+\cup \mathcal{W}_-)
\end{equation*} 
\end{lem}
\begin{proof}
Let us denote with $u^\#_+$ (resp. $u^\#_-$) the decreasing convex rearrangement of $u\arrowvert_{\Omega_+}$ (resp. $u\arrowvert_{\Omega_-}$). The P\'olya-Szeg\"o principle (\ref{poszin}) and properties of convex rearrangements provide
\begin{equation}
\label{eigmaj}
\lambda^T(\Omega) \geq\frac{\int_{\mathcal{W}_+} (H(\nabla u^\#_+))^2\ \text{d}s+\int_{\mathcal{W}_-} (H(\nabla u^\#_-))^2\ \text{d}s}{\int_{\mathcal{W}_+} (u^\#_+)^2\ \text{d}s+\int_{\mathcal{W}_-} ( u^\#_-)^2\ \text{d}s}
\end{equation}
and 
\begin{equation}
\label{zeavre}
\int_{\mathcal{W}_+} u^\#_+\ \text{d}s-\int_{\mathcal{W}_-} u^\#_-\ \text{d}s=\int_{\Omega^+} u\ \text{d}x+\int_{\Omega^-} u\ \text{d}x=\int_{\Omega} u\ \text{d}x=0.
\end{equation}
In view of (\ref{eigmaj}) and (\ref{zeavre}), we have the following inequality:
\begin{equation*}
\lambda^T(\Omega)\geq\lambda^*:=\inf_{\substack{ (f,g)\in H_0^1(\mathcal{W}_+)\times H_0^1(\mathcal{W}_-) \\ 
\int_{\mathcal{W}_+} f\ \text{d}s=\int_{\mathcal{W}_-} g\ \text{d}s}}\frac{\int_{\mathcal{W}_+} (H(\nabla f))^2\ \text{d}s+\int_{\mathcal{W}_-} (H(\nabla g))^2\ \text{d}s}{\int_{\mathcal{W}_+} f^2\ \text{d}s+\int_{\mathcal{W}_-} g^2\ \text{d}s}.
\end{equation*}
Using classical methods of calculus of variations, we can prove that this infimum is attained in $(f,g)$. Now, following the ideas of \cite[Sect. 3]{FH}, the function
\begin{equation*}
w=
\begin{sistema}
f\quad\text{in} \ \mathcal{W}_+\\
-g\quad\text{in} \ \mathcal{W}_-
\end{sistema}
\end{equation*}
satisfies
\begin{equation}
\label{twipro}
\begin{sistema}
-\text{div}(H(\nabla w)\nabla H(\nabla w))=\lambda^*w-\frac{1}{|\Omega|}\int_{\mathcal{W}_+\cup \mathcal{W}_-} \text{div}(H(\nabla w)\nabla H(\nabla w))\ \text{d}x\ \ \text{in}\ \mathcal{W}_+\cup \mathcal{W}_-\\
w=0\quad\text{on}\ \partial (\mathcal{W}_+\cup \mathcal{W}_-).
\end{sistema}
\end{equation}
This shows that $\lambda^*$ is an eigenvalue of the twisted problem (\ref{inftwi}) on $\mathcal{W}_+\cup \mathcal{W}_-$ and therefore, $\lambda^T (\Omega)\geq \lambda^*\geq\lambda^T (\mathcal{W}_+\cup \mathcal{W}_-)$.
\end{proof}
Throughout this Section, we investigate the first eigenvalue when $\Omega$ is the union of two disjoint Wulff sets, of radii $R_1\leq R_2$. Without loss of generality, we assume that the volume of $\Omega$ is such that
\begin{equation}
R_1^n+R_2^n=1
\end{equation}
and we denote by $\theta(R_1,R_2)$, the first positive root of equation
\begin{equation}
\label{eqcoeq}
R_1^n\frac{J_{\frac{n}{2}+1}\left(\theta\ R_1\right)}{J_{\frac{n}{2}-1}\left(\theta\ R_1\right)}+R_2^n\frac{J_{\frac{n}{2}+1}\left(\theta\ R_2\right)}{J_{\frac{n}{2}-1}\left(\theta\ R_2\right)}=0
\end{equation}
Now we recall a result given in \cite[Prop. 3.2]{FH}.

\begin{prop}
\label{compro}
There exists a constant $c_n<1$, depending on the dimension $n$, such that
\begin{itemize}
\item[(a)] if $R_1/R_2< c_n$, then $\lambda^T\left( \mathcal{W}_{R_1}\cup  \mathcal{W}_{R_2}\right)=\left(\frac{j_{\frac{n}{2},1}}{R_2}\right)^2$;\\
\item[(b)] if $R_1/R_2\geq c_n$, then $\lambda^T\left( \mathcal{W}_{R_1}\cup  \mathcal{W}_{R_2}\right)=\theta^2(R_1,R_2)$.
\end{itemize}
\end{prop}
\noindent Moreover, if we set $\theta^* =2^{1/n}j_{\frac{n}{2}-1,1}$, we obtain the following

\begin{prop}
\label{thesta}
The first positive root equation of {\rm (\ref{eqcoeq})} $\theta(R_1,R_2)$ satisfies
\begin{equation}
 \theta(R_1,R_2)\geq \theta^*,
 \end{equation}
for all $R_1$, $R_2\geq 0$.
\end{prop}
This result is proved in \cite[Lemma 3.3]{FH} when $\Omega$ has the same measure as the unit ball, but it can be obtained for all sets of finite measure. Now, we show the following isoperimetric inequality.

\begin{thm}
\label{isinth}
Let $\Omega$ be any bounded open set in $\R^n$. Then
\begin{equation}
\label{geisin}
\lambda^T(\Omega)\geq\lambda^T(\mathcal{W}_1 \cup \mathcal{W}_2)
\end{equation}
where $\mathcal{W}_1$ and $\mathcal{W}_2$ are two disjoint Wulff sets of measure $|\Omega|/2$. Equality holds if and only if $\Omega = \mathcal{W}_1\cup \mathcal{W}_2$.
\end{thm}
\begin{proof}
Thanks to Lemma \ref{lemudb}, it remains to prove that the union of two disjoint Wulff sets with the same measure gives the lowest possible value of $\lambda^T(\cdot)$ among unions of disjoint Wulff sets with given measure $|\Omega|$. Hence, we compute the first twisted eigenvalue of the union $\Omega$ of the Wulff sets $\mathcal{W}_{R_1}(x_1)$ and $\mathcal{W}_{R_2}(x_2)$, with $R_1 \leq R_2$. If we consider the eigenfunction that is zero on the smaller Wulff set and coincides with the first eigenfunction on the larger one, we trivially have $\lambda_1^T(\mathcal{W}_{R_1}(x_1) \cup \mathcal{W}_{R_2}(x_2))=\lambda_1^T(\mathcal{W}_{R_2}(x_2))$. 

Now we study the case in which the eigenfunction $u$ does not vanish on any of the two Wulff sets. We denote by $u_1$ and $u_2$ the functions that express $u$ respectively on $\mathcal{W}_{R_1}(x_1)$ and $\mathcal{W}_{R_2}(x_2)$. The proof of Lemma \ref{lemudb} shows that we can study only functions dependent on the radius of the Wulff set in which are defined. Therefore, in an abuse of notation, we consider functions such that $u_j(x) = u_j(H^o(x-x_j))$, for $j=1, 2$, and hence, instead of (\ref{twipro}), we can solve equivalently (see \cite{PII})
\begin{equation}
\label{ujwpro}
\begin{sistema}
u_j''(\rho)+ \frac{n-1}{\rho} u_j'(\rho)+\lambda^T u_j(\rho)=c, \ 0 < \rho < {R_j}\\
u_j'(0)=0, \ u_j(R_j)=0,
\end{sistema}
\end{equation}
for $j=1,2$, where $c=\frac{1}{|\Omega|}\int_\Omega \text{div}(H(\nabla u)\nabla H(\nabla u))\ \text{d}x$. Therefore, the solution $u$ of (\ref{ujwpro}) can be written in the form:
\begin{equation}
\label{discon}
u=
\begin{sistema}
u_1=c_1\left((H^o(x-x_1))^{1-\frac{n}{2}}J_{\frac{n}{2}-1}\left(\sqrt{\lambda^T}\ H^o(x-x_1)\right)\right.\\
\qquad\qquad\qquad\qquad\qquad\qquad\qquad\quad \left.-R_1^{1-\frac{n}{2}}J_{\frac{n}{2}-1}\left(\sqrt{\lambda^T}\ R_1\right)\right) \ \text{in} \ \mathcal{W}_{R_1}(x_1) \\
u_2=-c_2\left((H^o(x-x_2))^{1-\frac{n}{2}}J_{\frac{n}{2}-1}\left(\sqrt{\lambda^T}\ H^o(x-x_2)\right)\right.\\
\qquad\qquad\qquad\qquad\qquad\qquad\qquad\quad\left.-R_2^{1-\frac{n}{2}}J_{\frac{n}{2}-1}\left(\sqrt{\lambda^T}\ R_2\right)\right) \ \text{in} \ \mathcal{W}_{R_2}(x_2)
\end{sistema}
\end{equation}
Now we express the coupling condition $\int_\Omega u \ \text{d}x=0$ as
\begin{equation}
0=\int_{\mathcal{W}_{R_1}}u_1\ \text{d}x+\int_{\mathcal{W}_{R_2}} u_2\ \text{d}x,
\end{equation}
and hence we obtain
\begin{equation*}
\begin{split}
0&=c_1\left( \gamma_n\int_0^{R_1}J_{\frac{n}{2}-1}\left(\sqrt{\lambda^T}\ \rho \right) \rho^\frac{n}{2}\ \text{d}\rho- \kappa_n R_1^{\frac{n}{2}+1}J_{\frac{n}{2}-1}\left(\sqrt{\lambda^T}\ R_1\right)\right)\\
& \qquad -c_2\left( \gamma_n\int_0^{R_2}J_{\frac{n}{2}-1}\left(\sqrt{\lambda^T}\ \rho\right) \rho^\frac{n}{2}\ \text{d}\rho- \kappa_n R_2^{\frac{n}{2}+1}J_{\frac{n}{2}-1}\left(\sqrt{\lambda^T}\ R_2\right)\right).
\end{split}
\end{equation*}
We use classical properties of Bessel functions \cite{W}, namely
\begin{equation*}
\int_0^RJ_{\frac{n}{2}-1}\left(k r\right)r^\frac{n}{2}dr=\frac{1}{k}R^\frac{n}{2}J_\frac{n}{2}(kr) \quad \text{and} \quad \frac{n}{kr}J_\frac{n}{2}(kr)-J_{\frac{n}{2}-1}(kr)=J_{\frac{n}{2}+1}(kr),
\end{equation*}
together with $\gamma_n=n\kappa_n$, where $\gamma_n$  is the generalized perimeter of $\mathcal{W}$, to get
\begin{equation}
c_1R_1^{\frac{n}{2}+1}J_{\frac{n}{2}+1}\left(\sqrt{\lambda^T}\ R_1\right)-c_2R_2^{\frac{n}{2}+1}J_{\frac{n}{2}+1}\left(\sqrt{\lambda^T}\ R_2\right)=0.
\end{equation}
Hence it is possible to take
\begin{equation}
\label{c_1c_2}
c_1=R_2^{\frac{n}{2}+1}J_{\frac{n}{2}+1}\left(\sqrt{\lambda^T}\ R_2\right)\ \ \text{and}\ \ c_2=R_1^{\frac{n}{2}+1}J_{\frac{n}{2}+1}\left(\sqrt{\lambda^T}\ R_1\right)
\end{equation}
in (\ref{discon}). 

Now we want that the constant $c$ in (\ref{ujwpro}) is the same for $j=1$ and for $j=2$. This automatically implies that this constant $c$ coincides with the average of the anisotropic laplacian computed on $u$. Since 
\begin{equation*}
\begin{split}
\text{div}(H(\nabla u_1)\nabla H(\nabla u_1))=-c_1\lambda^T (H^o(x-x_1))^{1-\frac{n}{2}}J_{\frac{n}{2}-1}\left(\sqrt{\lambda^T}\ H^o(x-x_1) \right)\\
\text{div}(H(\nabla u_2)\nabla H(\nabla u_2))=c_2\lambda^T (H^o(x-x_2))^{1-\frac{n}{2}}J_{\frac{n}{2}-1}\left(\sqrt{\lambda^T}\ H^o(x-x_2)\right),
\end{split}
\end{equation*}
we have
\begin{equation*}
\begin{split}
c=\text{div}(H(\nabla u_1) & \nabla H(\nabla u_1))+\lambda^Tu_1=- c_1 \lambda^T  R_1^{1-\frac{n}{2}}J_{\frac{n}{2}-1}\left(\sqrt{\lambda^T}\ R_1\right)\\
c=\text{div}(H(\nabla u_2) & \nabla H(\nabla u_2))+\lambda^Tu_1=c_2  \lambda^T  R_2^{1-\frac{n}{2}}J_{\frac{n}{2}-1}\left(\sqrt{\lambda^T}\ R_2\right).\\
\end{split}
\end{equation*}
Comparing this two relations and taking in account (\ref{c_1c_2}), if we set $\lambda^T( \mathcal{W}_{R_1}(x_1)\cup  \mathcal{W}_{R_2}(x_2))=\theta^2$, the condition $-c+c=0$ gives the equation (\ref{eqcoeq}). Now we observe that, in the case that $R_1/R_2<c_n$, by Proposition \ref{compro}\emph{(a)} and by the inequality $j_{\frac{n}{2},1}>2^{1/n}j_{\frac{n}{2}-1,1}$ \cite[Cor. A.2]{FH}, we have
\begin{equation}
\lambda^T( \mathcal{W}_{R_1}(x_1)\cup  \mathcal{W}_{R_2}(x_2))\geq \left( \frac{j_{\frac{n}{2},1}}{R_2} \right)^2\geq\left( j_{\frac{n}{2},1} \right)^2>\left( 2^{1/n}j_{\frac{n}{2}-1,1}\right)^2=\theta^{*2}.
\end{equation}
If $R_1/R_2\geq c_n$, by Proposition \ref{compro}\emph{(b)} and Proposition \ref{thesta}, we have
\begin{equation}
\lambda^T( \mathcal{W}_{R_1}(x_1)\cup  \mathcal{W}_{R_2}(x_2)) = (\theta(R_1,R_2))^2\geq\theta^{*2}.
\end{equation}
Therefore, in both case, we obtain that $\lambda^T( \mathcal{W}_{R_1}(x_1)\cup  \mathcal{W}_{R_2}(x_2)) \geq\theta^{*2}$ and since $\theta^*$ is the value of $\lambda^T(\Omega)$ computed on two Wulff sets with the same measure, this conclude the proof.
\end{proof}

\section{The Nonlocal Problem}
\label{secpro}
The aim of this Section is to prove Theorem \ref{thmsat}. We start by showing some preliminary results.
\begin{thm}
\label{twithe}
Let $\Omega$ be an open bounded set in $\R^n$, then there exists a positive value of $\alpha$ such that the corresponding first eigenvalue $\lambda(\alpha,\Omega)$ is greater or equal than $\frac{2^{2/n}\kappa_n^{2/n}j^2_{n/2-1,1}}{|\Omega|^{2/n}}$.
\end{thm}
\begin{proof}
We first observe that $\lambda(\alpha,\Omega)$ is bounded, indeed
\begin{equation*}
\lim_{\alpha\rightarrow+\infty}\lambda(\alpha,\Omega)\leq\min_{\substack{u\in H_0^1(\Omega) \\ u \not\equiv 0}}\left\{\frac{\int_\Omega (H(\nabla u))^2\ \text{d}x}{\int_\Omega u^2\ \text{d}x}, \int_\Omega u \ \text{d}x=0\right\}=\lambda^T(\Omega).
\end{equation*}
Compactness arguments show that there exists a sequence of eigenfunctions $u_\alpha$, $\alpha\rightarrow + \infty$, with norm in $L^2(\Omega)$ equal to $1$, weakly converging in $H_0^1(\Omega)$ and strongly in $L^2(\Omega)$ to a function $u$. Obviously $\int_\Omega u_\alpha\ \text{d}x\rightarrow\int_\Omega u \ \text{d}x=0$, as $\alpha\rightarrow+\infty$ (this limit exists by compactness) and hence, by the lower semicontinuity \cite[Th. 4.5]{G}
\begin{equation*}
\lim_{\alpha\rightarrow+\infty}\lambda(\alpha,\Omega)\geq\inf_{\substack{u\in H_0^1(\Omega) \\ u \not\equiv 0}}\left\{\frac{\int_\Omega (H(\nabla u))^2\ \text{d}x}{\int_\Omega u^2\ \text{d}x}, \int_\Omega u \ \text{d}x=0\right\}=\lambda^T(\Omega).
\end{equation*}
In Theorem \ref{isinth} we have proved that the last term is greater or equal than the first eigenvalue on two disjoint Wulff sets of equal radii. Therefore, by Proposition \ref{eigalp}\emph{(c)}, the result follows.
\end{proof}

\begin{prop}
\label{baldes}
If $\alpha >0$ and $\Omega$ is a bounded, open set in $\R^n$ which is not union of two disjoint Wulff sets. Then there exist $\mathcal{W}_{R_1}$ and $\mathcal{W}_{R_2}$ disjoint such that $|\mathcal{W}_{R_1}\cup \mathcal{W}_{R_2}|=|\Omega|$ and
\begin{equation}
\lambda(\alpha,\Omega)>\lambda(\alpha,\mathcal{W}_{R_1}\cup \mathcal{W}_{R_2})=\min_{\substack{  A=\mathcal{W}_{R_1} \cup  \mathcal{W}_{R_2}\\ |A|=|\Omega|}}\lambda(\alpha,A).
\end{equation}
\end{prop}
\begin{proof}
Let $u$ be an eigenfunction of (\ref{infquo}), $\Omega_\pm$ be defined as in (\ref{deo+o-}), $u_\pm=u|_{\Omega\pm}$ and $\Omega_\pm^\#$ be the Wulff sets with the same measure as $\Omega_\pm$. Using (\ref{poszin}), it is easy to show that
\begin{equation}
\label{geqine}
\lambda(\alpha,\Omega)\geq
\min_{\substack{  A=\mathcal{W}_{R_1} \cup  \mathcal{W}_{R_2}\\ |A|=|\Omega|}}\lambda(\alpha,A).
\end{equation}
Indeed, we have
\begin{equation}
\label{pr31in}
\begin{split}
& \lambda (\alpha,\Omega)=\frac{\int_\Omega (H(\nabla u))^2\ \text{d}x+\alpha\left(\int_\Omega u\ \text{d}x\right)^2}{\int_\Omega u^2\ \text{d}x}\\
&\geq \frac{\int_{\Omega_+^\#} (H(\nabla (u_+)^\#))^2\ \text{d}s+\int_{\Omega_-^\#} (H(\nabla (u_-)^\#))^2\ \text{d}s+\alpha\left(\int_{\Omega_+^\#} (u_+)^\# \ \text{d}s- \int_{\Omega_-^\#} (u_-)^\#\ \text{d}s\right)^2}{\int_{\Omega_+^\#} (u_+)^{\substack{\star^2}}\ \text{d}s + \int_{\Omega_-^\#} (u_-)^{\substack{\star^2}}\ \text{d}s}\\
&\geq \min_{(f,g)\in H^0_1(\Omega_+^\#)\times H^0_1(\Omega_+^\#)}\frac{\int_{\Omega_+^\#} (H(\nabla f))^2\ \text{d}s+\int_{\Omega_-^\#} (H(\nabla g))^2\ \text{d}s+\alpha\left(\int_{\Omega_+^\#} f\ \text{d}s - \int_{\Omega_-^\#} g\ \text{d}s\right)^2}{\int_{\Omega_+^\#} f^2 \ \text{d}s+ \int_{\Omega_-^\#} g^2\ \text{d}s}\\
&= \lambda(\alpha,\Omega_+^\#\cup\Omega_-^\#)\\
&\geq \inf_{\substack{  A=\mathcal{W}_{R_1} \cup  \mathcal{W}_{R_2}\\ |A|=|\Omega|}}\lambda(\alpha,A)\\
\end{split}
\end{equation}
Let us prove that, actually, the inequality (\ref{geqine}) is strict. Suppose, by contradiction that (\ref{geqine}) holds as an equality. In particular, from (\ref{pr31in}) we have
\begin{equation}
\lambda(\alpha,\Omega)
= \lambda(\alpha,\Omega_+^\#\cup\Omega_-^\#),
\end{equation}
hence, by Theorem \ref{psieth}, we deduce that $\Omega_+^\#$ and $\Omega_-^\#$ are Wulff sets. Then, we may have two cases:
\begin{itemize}
\item[(i)]$\Omega=\Omega_+^\#\cup\Omega_-^\#$,
\item[(ii)]$|\Omega_+| + |\Omega_-|<|\Omega|$.
\end{itemize}
In the first case, we have immediately a contradiction because, by hypothesis, $\Omega$ is not a union of two Wulff sets.\\
In the second case, we observe that eigenfunction $u$ vanishes on a set of positive measure and, by (\ref{eigpro}), it has zero average. Using the strict monotonicity of the Dirichlet eigenvalue with respect to homotheties, we again reach a contradiction since $\lambda(\alpha,\Omega_+^\#\cup\Omega_-^\#)>\inf_{\substack{  A=\mathcal{W}_{R_1} \cup  \mathcal{W}_{R_2}\\ |A|=|\Omega|}}\lambda(\alpha,A)$. Therefore, we have
\begin{equation}
\lambda(\alpha,\Omega)> \inf_{\substack{  A=\mathcal{W}_{R_1} \cup  \mathcal{W}_{R_2}\\ |A|=|\Omega|}}\lambda(\alpha,A).
\end{equation}
Finally, the compactness of family of disjoint pair of Wulff sets and the continuity of $\lambda (\alpha, \Omega)$ with respect to uniform convergence of the domains gives the existence of the set $A=\mathcal{W}_{R_1} \cup  \mathcal{W}_{R_2}$ (see e.g. \cite{BB}, \cite{HP}) .
\end{proof}

\begin{rem}
\label{decrem}
Before showing the next result, let us observe that when $\Omega$ reduces to the union of two Wulff sets $\mathcal{W}_{R_1}\cup\mathcal{W}_{R_2}$, then problem (\ref{eigpro}) becomes 
\begin{equation}
\label{decpro}
\begin{sistema}
-\text{div} (H(\nabla u) \nabla H(\nabla u))+\alpha\left(\int_{\mathcal{W}_{R_1}} u \ \text{d}x+\int_{\mathcal{W}_{R_2}} v\ \text{d}x\right)=\lambda u \quad\text{in} \ \mathcal{W}_{R_1}\\
-\text{div} (H(\nabla v) \nabla H(\nabla v))+\alpha\left(\int_{\mathcal{W}_{R_1}} u \ \text{d}x+\int_{\mathcal{W}_{R_2}} v\ \text{d}x\right)=\lambda v \quad\text{in} \ \mathcal{W}_{R_2}\\
u=0\quad\text{on}\ \partial \mathcal{W}_{R_1}, \quad 
v=0\quad\text{on}\ \partial \mathcal{W}_{R_2}.
\end{sistema}
\end{equation}
\end{rem}

\begin{prop}
\label{proalp}
Let $\Omega$ be the union of two disjoint Wulff sets $\mathcal{W}_{R_1}$ and $\mathcal{W}_{R_2}$ such that $\kappa_n(R_1^n+R_2^n)=|\Omega|$. Then, for every $\eta\in\left(\frac{\kappa_n^{2/n}j^2_{n/2-1,1}}{|\Omega|^{2/n}},\frac{2^{2/n}\kappa_n^{2/n}j^2_{n/2-1,1}}{|\Omega|^{2/n}}\right)$ and for every $R_1$, $R_2 \geq 0$, there exists a unique value of $\alpha$, denoted by $\alpha_\eta$, given by 
\begin{equation}
\label{alpmuu}
\frac{1}{\alpha_\eta}=\frac{\kappa_n(R_1^n+R_2^n)}{\eta}-\frac{n\kappa_n}{\eta^{3/2}}\left[R_1^{n-1}\frac{J_{n/2}(\sqrt{\eta}R_1)}{J_{n/2-1}(\sqrt{\eta}R_1)}+R_2^{n-1} \frac{J_{n/2}(\sqrt{\eta}R_2)}{J_{n/2-1}(\sqrt{\eta}R_2)}  \right],
\end{equation}
such that
$\eta=\lambda(\alpha_\eta,\mathcal{W}_{R_1}\cup \mathcal{W}_{R_2})$.
\end{prop}
\begin{proof}
In view of Proposition \ref{baldes}, problem (\ref{eigpro}) reduces to (\ref{decpro}). Then, we easily verify that the functions
\begin{equation*}
u=R_2^{1-\frac{n}{2}}J_{\frac{n}{2}-1}\left(\sqrt{\eta}R_2\right)\left[\left(H^o(x)\right)^{1-\frac{n}{2}}  J_{\frac{n}{2}-1}\left(\sqrt{\eta} \ H^o(x)\right)  - R_1^{1-\frac{n}{2}}J_{\frac{n}{2}-1}\left(\sqrt{\eta}\ R_1\right)\right] 
\end{equation*}
and
\begin{equation*}
v=R_1^{1-\frac{n}{2}}J_{\frac{n}{2}-1}\left(\sqrt{\eta}R_1\right)\left[\left(H^o(x)\right)^{1-\frac{n}{2}}  J_{\frac{n}{2}-1}\left(\sqrt{\eta} \ H^o(x) \right)  - R_2^{1-\frac{n}{2}}J_{\frac{n}{2}-1}\left(\sqrt{\eta}\ R_2\right) \right].
\end{equation*}
solve problem (\ref{decpro}). Now, we show that, for all $R_1$, $R_2$ and $\eta$ as in the hypothesis, there exists at least one value of $\alpha$ such that problem (\ref{decpro}) admits a non trivial solution. Indeed, the eigenvalue $\lambda(\alpha,\mathcal{W}_{R_1}\cup\mathcal{W}_{R_2})$ is clearly unbounded from below as $\alpha\rightarrow - \infty$ and, by Theorem \ref{twithe}, is larger than $\frac{2^{2/n}\kappa_n^{2/n}j^2_{n/2-1,1}}{|\Omega|^{2/n}}$ as $\alpha\rightarrow + \infty$. Hence the continuity and the monotonicity of $\lambda(\alpha,\mathcal{W}_{R_1}\cup\mathcal{W}_{R_2})$ with respect to $\alpha$ implies that when $\alpha = \alpha_\eta$, then $\eta$ is the first eigenvalue of problem (\ref{decpro}). This value of $\alpha$ is unique, indeed, arguing by contradiction, if for some $\eta$, there exists another value $\alpha\neq\alpha_\eta$ such that $\eta$ is the first eigenvalue of problem (\ref{decpro}), then by Proposition \ref{anoalp}, $\eta$ is also an eigenvalue of the local problem and the corresponding eigenfunction has zero average in $\mathcal{W}_{R_1}\cup\mathcal{W}_{R_2}$. Hence, if these Wulff sets have the same measure, then the eigenvalue $\eta$ is, by Proposition \ref{eigalp}\emph{(c)}, equal to $\frac{2^{2/n}\kappa_n^{2/n}j^2_{n/2-1,1}}{|\Omega|^{2/n}}$ and this contradicts the fact that $\eta<\frac{2^{2/n}\kappa_n^{2/n}j^2_{n/2-1,1}}{|\Omega|^{2/n}}$. Otherwise, if the sets do not have the same measure, by Proposition \ref{tworem}\emph{(a)}, the first eigenfunction is identically zero on the smaller set and it does not change sign on the larger one; this is in contradiction with the fact that the eigenfunction is not trivial and has zero average.
\end{proof}
\noindent \emph{\bf Proof of Theorem \ref{thmsat}.}
Let us firstly analize the case of nonpositive $\alpha$. Let $u$ be an eigenfunction, by (\ref{poszin}) we have
\begin{equation*}
\lambda (\alpha,\Omega)=\mathscr{Q}_\alpha (u,\Omega)\geq\mathscr{Q}_\alpha (|u|,\Omega)\geq \mathscr{Q}_\alpha (u^\#,\Omega^\#)\geq \lambda (\alpha,\Omega^\#).
\end{equation*}
By Proposition \ref{sigthm}, we can say that $u$ is positive; therefore $\Omega$ coincides with the set \mbox{$\{ x\in\Omega : u(x)>0 \}$} that, by Theorem \ref{psieth}, is equivalent to a Wulff set. Therefore the equality case is proved.

Now, we study the case of positive value of $\alpha$. In view of Proposition \ref{baldes} we can restrict our study to the case of two disjoint Wulff sets, of radii $R_1$ and $R_2$, whose union has the same measure of $\Omega$. By Proposition \ref{eigalp}\emph{(b)-(c)}, the first eigenvalue is greater than $\frac{\kappa_n^{2/n}j^2_{n/2-1,1}}{|\Omega|^{2/n}}$ and is lower than the first eigenvalue computed on two Wulff sets with the same measure, that is $\frac{2^{2/n}\kappa_n^{2/n}j^2_{n/2-1,1}}{|\Omega|^{2/n}}$. Hence, we can restrict our study to the eigenvalues in the range $\left( \frac{ \kappa_n^{2/n}j^2_{n/2-1,1}} {|\Omega|^{2/n}},\frac{2^{2/n}\kappa_n^{2/n}j^2_{n/2-1,1}}{|\Omega|^{2/n}}\right)$. Now, let us observe that if $\Omega$ is a Wulff set and $\alpha = \alpha_c|\Omega|^{-1-2/n}$, then, from  (\ref{alpmuu}) with $R_2=0$, $\lambda(\alpha,\Omega^\#)=\frac{2^{2/n}\kappa_n^{2/n}j^2_{n/2-1,1}}{|\Omega|^{2/n}}$. Therefore $\alpha = \alpha_c|\Omega|^{-1-2/n}$ is a critical value of $\alpha$ because the first eigenvalue on $\Omega^\#$ coincides with the first eigenvalue on the union of two disjoint Wulff sets of equal radii.

We firstly analyze the subcritical cases ($0 < \alpha < \alpha_c|\Omega|^{-1-2/n}$). Thanks to Proposition \ref{baldes}, it remains to prove that if $\Omega$ is union of two non negligible disjoint Wulff sets, then $\lambda(\alpha,\Omega)>\lambda(\alpha,\Omega^\#)$. Therefore, by Proposition \ref{proalp}, this is equivalent to say that, for any $\eta\in\left(\frac{\kappa_n^{2/n}j^2_{n/2-1,1}}{|\Omega|^{2/n}},\frac{2^{2/n}\kappa_n^{2/n}j^2_{n/2-1,1}}{|\Omega|^{2/n}}\right)$,  $\alpha_\eta$ attains its maximum if and only if $R_1$ or $R_2$ vanishes, with the constraint $\kappa_n(R_1^n+R_2^n)=|\Omega|$. This is proved in \cite[Prop. 3.4]{BFNT} with $\kappa_n$ instead of $\omega_n$ using Bessel function properties. 

The continuity of $\lambda(\alpha, \Omega)$ with respect to $\alpha$ for subcritical values yields $\lambda(\alpha_c, \Omega) \geq \lambda(\alpha_c, \Omega^\#) $. Hence, for supercritical values ($\alpha>\alpha_c|\Omega|^{-1-2/n}$), using the monotonicity of $\lambda(\alpha,\Omega)$ with respect to $\alpha$, we have 
\begin{equation}
\label{superi}
\lambda(\alpha, \Omega) \geq \lambda(\alpha_c, \Omega) \geq \lambda(\alpha_c, \Omega^\#) = \frac{2^{2/n}\kappa_n^{2/n}j^2_{n/2-1,1}}{|\Omega|^{2/n}}.
\end{equation}

If the inequalities in (\ref{superi}) hold as equalities, then $\Omega$ is the union of two disjoint Wulff sets of same measure. Indeed, by Proposition \ref{eigalp}\emph{(a)}, the first inequality is strict only when the eigenfunction relative to $\lambda(\alpha,\Omega)$ has nonzero average, that is when the two Wulff sets have different radii. Hence also the equality case follows and this conclude the proof of the Theorem \ref{thmsat}.


\end{document}